\documentclass[3p,10pt]{elsarticle}

\usepackage[utf8x]{inputenc}
\usepackage{amsmath,amsthm}
\usepackage{amssymb}
\usepackage{array,blkarray}
\usepackage{arydshln}
\usepackage{multirow,bigdelim}
\usepackage{color}
\usepackage[all]{xy}

\newtheorem{theorem}{Theorem}[section]

\newtheorem{lemma}[theorem]{Lemma}
\newtheorem{prop}[theorem]{Proposition}
\newtheorem{cor}[theorem]{Corollary}
\newtheorem*{conjecture}{Conjecture}
\newtheorem{remark}[theorem]{Remark}
\newtheorem*{theorem*}{Theorem}

\theoremstyle{definition}

\newtheorem*{defn*}{Definition}

\def\XXint#1#2#3{{\setbox0=\hbox{$#1{#2#3}{\int}$}
     \vcenter{\hbox{$#2#3$}}\kern-.5\wd0}}

\newcommand{\Aut}{\textup{Aut}}

\newcommand{\Cl}{\mathcal{C}\negthinspace\ell}

\newcommand{\cyclic}[1]{C_{#1}}

\newcommand{\Epi}{\textup{Epi}}

\newcommand{\ffield}[1]{\mathbb{F}_{#1}}

\newcommand{\Gal}{\textup{Gal}}

\newcommand{\grindex}[2]{\left(#1\negthinspace:\negthinspace #2\right)}
\newcommand{\grossO}[1]{O\left(#1\right)}
\newcommand{\grossOmega}[1]{\Omega\left(#1\right)}

\newcommand{\Hom}{\textup{Hom}}

\newcommand{\iso}{\simeq}

\newcommand{\mc}[1]{\mathcal{#1}}
\newcommand{\mf}[1]{\mathfrak{#1}}
\newcommand{\mG}{w(G)}
\newcommand{\mmG}{m}
\newcommand{\N}[1]{\Norm{\mathfrak{#1}}}
\newcommand{\Norm}[1]{\big\|#1\big\|}

\newcommand{\open}[2]{1+\mf{#1}^{#2}}

\newcommand{\operp}{\bigcirc\kern-1.17em\perp}
\newcommand{\val}[1]{v_{\mf{#1}}}
\newcommand{\ord}[2]{\val{#1}(#2)}

\newcommand{\QQ}{\mathbf{Q}}

\newcommand{\tupel}[1]{\boldsymbol{#1}}

\newcommand{\ZZ}{\mathbf{Z}}

\date{2013}
\renewcommand{\thefootnote}{\fnsymbol{footnote}}

\begin{document}

\begin{frontmatter}

\title{Distribution of Artin-Schreier-Witt Extensions}
\author{Thorsten Lagemann\footnote{
ERGO Versicherungsgruppe, Aachener Stra\ss e 300, 50933 K\"oln, Germany. t.lagemann@ergo.de}}
\journal{}

\begin{abstract}
The article at hand contains exact asymptotic formulas for the distribution of conductors of abelian $p$-extensions of global function fields of characteristic~$p$. These yield a new conjecture for the distribution of discriminants fueled by an interesting lower bound.
\end{abstract}

\end{frontmatter} 

\setcounter{footnote}{0}
\renewcommand{\thefootnote}{\arabic{footnote}}

\section{Main results}

Let $F$ be a global function field of characteristic~$p$, that is a transcendental extension of degree $1$ over some finite field $\ffield{q}$ of cardinality $q$, and let $\ffield{q}$ be algebraically closed in $F$. For an extension $E/F$, let $\mf{f}(E/F)$ be its relative conductor, and $\Norm{\mf{f}(E/F)} = q^{\deg \mf{f}(E/F)}$ its absolute norm. The object of interest is the counting function
$$ C(F,G;X) = \left|\Big\{ E/F : \Gal(E/F) \iso G,\, \Norm{\mf{f}(E/F)} \leq X \Big\}\right| $$
of field extensions $E/F$ in a fixed algebraic closure of $F$ with given finite Galois group $G$ and bounded conductor. The values of this function coincides with finite coefficient sums of the Dirichlet series 
$$ \Phi(F,G;s) = \sum\limits_{\Gal(E/F)\iso G} \Norm{\mf{f}(E/F)}^{-s}, $$
generated by the extensions $E/F$ with Galois group $G$. For the article at hand, we are interested in field extensions with abelian Galois group of exponent $p^e$. These extensions are called Artin-Schreier-Witt extensions.

\begin{theorem}\label{main1} Let $F$ be a global function field of characteristic~$p$ and $G$ be the finite abelian $p$-group with exponent $p^e$ and $p^i$-ranks $r_i(G) = \log_p \grindex{p^{i-1}G}{p^iG}$. Let 
$$ \alpha_p(G) = \frac{1 + (p-1) \sum\nolimits_{i=1}^e p^{e-i} r_i(G)}{p^e} $$
and
$$ \beta(F,G) = \begin{cases} 
  p^e - 1, & \text{if $G$ is cyclic}, \\
  p^{e-f}, & \text{elsewise, where $1\leq f\leq e$ is minimal such that $p^fG$ is cyclic}.               
                \end{cases} 
$$
Then the Dirichlet series $\Phi(F,G;s)$ has convergence abscissa $\alpha_p(G)$, and it possesses a meromorphic continuation, such that $\Phi(F,G;s) \prod\nolimits_{l=2}^{p^e} \zeta_F\big(ls - \sum\nolimits_{i=1}^e w_i(l)r_i(G)\big)$ is holomorphic for $\Re(s) > \alpha_p(G) - 1/(2p^e)$, where $\zeta_F(s)$ denotes the zeta function of $F$ and $w_i(l) = \lfloor (l-1)/p^{i-1} \rfloor - \lfloor (l-1)/p^i \rfloor$. The periodic pole $s = \alpha_p(G)$ with period $2\pi i/\log(q)$ has order $\beta(F,G)$, whereas the order of every other pole on the axis $\Re(s) = \alpha_p(G)$ does not exceed $\beta(F,G)$.  
\end{theorem}

The series $\Phi(F,G;s)$ is a power series in $t=q^{-s}$, which is convergent for $|t|<q^{-\alpha_p(G)}$, and its poles are located at $t=\xi q^{-\alpha_p(G)}$ for finitely many roots of unity $\xi$. If $c_n$ denotes the $n$-th power series coefficient of $\Phi(F,G;s)$, we obtain $C(F,G;q^m)=\sum\nolimits_{n=0}^m c_n$. By an application of the Cauchy integral formula, which is elaborated in Theorem A.5 \cite{Lagemann2012}, this coefficient sum has the following asymptotic behaviour.

\begin{theorem}\label{main2} Let $F$ be a global function field of characteristic~$p$ and $G$ be a finite abelian $p$-group. 
Then the number of Artin-Schreier-Witt extensions $E/F$ with Galois group $G$ has asymptotic\footnote{Hereby the asymptotic equivalence $f(X) \sim g(X)$ means that there are integers $\ell \geq 1$ and $e \geq 0$ such that $\lim\nolimits_{n\rightarrow\infty} f(q^{\ell n +e})/g(q^{\ell n + e}) = 1$ holds. The main reason for introducing this slightly weaker definition is, that the counting functions $f(X)$ in question are step functions in $X=q^n$, where $n$ runs over some arithmetic progression. Hence, the asymptotic quotient of $f(X)/g(X)$ would not exist for continuous functions $g(X)$. In order to compare function field asymptotics with number field asymptotics, we use this redefinition regarding the other Landau symbols as well. See \cite{Lagemann2012} for comments and examples.}
$$ C(F,G;X) \sim \gamma(F,G) X^{\alpha_p(G)} \log(X)^{\beta(F,G)-1} $$
with some explicitly computable constant $\gamma(F,G) > 0$.
\end{theorem}

Theorem \ref{main1} and Theorem \ref{main2} are generalisations of the results of my first article \cite{Lagemann2012} on the conductor distribution of abelian $p$-extensions in characteristic $p$. The computation of $\gamma(F,G)$ will not be generalised, examples for $G = \cyclic{p}$ or $G= \cyclic{2}^r$ can be found in \cite{Lagemann2012}. We will tightly follow the lead of \cite{Lagemann2012} by generalising the respective proof steps. The main ingredients for the proof are fivefold. We first need to compute the $p$-signature of ray class groups, which formula is named by \textit{Hasse}. Then the number of $G$-extensions $E/F$ with given module $\mf{m}$ is given by a \textit{Delsarte} formula for the number of subgroups with given quotient, which itself is recapped in Appendix A. Then we require a \textit{M\"obius} formula in order to compute the number $c_{\mf{m}}$ of $G$-extensions $E/F$ with conductor $\mf{m}$ from the number of $G$-extensions with module $\mf{m}$. In addition, the latter two formula 
provide a nice local-global principle for $c_{\mf{m}}$, which yield a decomposition of $\Phi(F,G;s)$ in sum of \textit{Euler} 
products. Finally, we obtain the asymptotic behaviour of $C(F,G;X)$ by a \textit{Cauchy-Tauber} formula. Similar to \cite{Lagemann2012}, we explore the consequences of Theorem \ref{main2} on the distribution of discriminants. We obtain a lower and an upper bound respectively, where the lower bound leads to a conjecture for the asymptotic behaviour of discriminants of Artin-Schreier-Witt extensions.

\section{Consequences on the distribution of discriminants}

For given global function field $F$ and finite group $G$, consider the number of field extensions $E/F$ in a fixed algebraic closure of $F$ with Galois group $G$ and bounded discriminant $\Norm{\mf{d}(E/F)} = q^{\deg \mf{d}(E/F)}$, that is the function
$$ Z(F,G;X) = \left|\Big\{ E/F : \Gal(E/F) \iso G,\, \Norm{\mf{d}(E/F)} \leq X \Big\}\right|. $$ 
For the finite abelian $p$-group $G$ with exponent $p^e$ and $p^i$-ranks $r_i(G) = \log_p \grindex{p^{i-1}G}{p^iG}$, let
$$ a_p(G) =  \frac{\alpha_p(G)}{\mG} =  \frac{1 + (p-1) \sum\nolimits_{i=1}^e p^{e-i} r_i(G)}{p \sum\nolimits_{i=1}^e p^{e-i} (|p^{i-1}G| - |p^iG|)}  $$
with
$$ \mG = \sum\limits_{i=0}^{e-1} p^{-i} \Big(\big|p^{i}G\big| - \big|p^{i+1}G\big|\Big).  $$

\begin{lemma}\label{FuehrerDiskriminantenFormelAbschaetzung} Let $E/F$ be a $G$-extension with conductor $\mf{f}(E/F)$ and discriminant $\mf{d}(E/F)$, and let $\mf{p}$ a prime supporting these modules. Then it has discriminant exponent
$$ \val{p}(\mf{d}(E/F)) \leq  \sum\limits_{i=0}^{e-1} \left\lceil \frac{\val{p}(\mf{f}(E/F))}{p^{i}} \right\rceil \Big( \left|p^{i} G \right| - \left|p^{i+1} G \right| \Big) ,  $$
where $\lceil x \rceil$ denotes the least integer larger or equal to $x$.
\end{lemma}
\begin{proof} Let $E_{\mf{P}}/F_{\mf{p}}$ be any localisation of $E/F$, $G_{\mf{p}}$ its decomposition group, and let $\hat{G}_{\mf{p}}$ be the group of characters $\chi : F_{\mf{p}}^{\times} \rightarrow \QQ/\ZZ$ vanishing on the norm group $N(E_{\mf{P}}^{\times}) \leq F_{\mf{p}}^{\times}$. For $\chi \in \hat{G}_{\mf{p}}$, let $f_{\mf{p}}(\chi)$ be the conductor exponent of its kernel in $F_{\mf{p}}^{\times}$, namely the least integer $m\geq 0$ such that $\chi$ vanishes on the $m^{\textup{th}}$ unit group $\langle 1 + \mf{p}^m \rangle \leq F_{\mf{p}}^{\times}$. By the F\"uhrerdiskriminantenformel\footnote{E.g. see formula~(5.9), p. 230 in \cite{KochAlgebraicNumberTheory}.} for abelian extensions, the global discriminant exponent of $\mf{p}$ is given by 
$$ \val{p}(\mf{d}(E/F)) = \grindex{G}{G_{\mf{p}}} \sum\limits_{\chi\in\hat{G}_{\mf{p}}} f_{\mf{p}}(\chi). $$ 
For any $i\geq 0$, we have
$$ f_{\mf{p}}(p^i\chi) \leq \left\lceil \frac{f_{\mf{p}}(\chi)}{p^i} \right\rceil. $$
This follows from the subgroup relation $ \langle 1 + \mf{p}^n \rangle^{p} \leq \langle 1 + \mf{p}^{pn} \rangle $,\footnote{E.g. see Proposition 5.6, p. 14 in \cite{Fesenko2002}.} 
and $ p\chi(\langle 1 + \mf{p}^n\rangle) = \chi(\langle 1 + \mf{p}^n\rangle^{p}) \leq \chi(\langle 1 + \mf{p}^{pn}\rangle). $ Then we obtain
$$ f_{\mf{p}}(\chi) \leq \left\lceil \frac{\val{p}(\mf{f}(E/F))}{p^i} \right\rceil $$
for $\chi\in p^i\hat{G}_{\mf{p}}$, and 
$$ \sum\limits_{\chi\in\hat{G}_{\mf{p}}} f_{\mf{p}}(\chi) \leq \sum\limits_{i=0}^{e-1} a_i \Big( \big|p^{i} \hat{G}_{\mf{p}} \big| - \big|p^{i+1} \hat{G}_{\mf{p}} \big| \Big),  $$
by setting $a_i = \lceil \val{p}(\mf{f}(E/F))/p^i \rceil$. With the F\"uhrerdiskriminantenformel above, we obtain
$$ \val{p}(\mf{d}(E/F)) \leq \grindex{G}{G_{\mf{p}}} \sum\limits_{i=0}^{e-1} a_i \Big( \big|p^{i} \hat{G}_{\mf{p}} \big| - \big|p^{i+1} \hat{G}_{\mf{p}} \big| \Big). $$
Now Lemma \ref{FuehrerDiskriminantenFormelAbschaetzung} follows by validating
$$ \grindex{G}{G_{\mf{p}}} \sum\limits_{i=0}^{e-1} a_i \Big( \big|p^{i} \hat{G}_{\mf{p}} \big| - \big|p^{i+1} \hat{G}_{\mf{p}} \big| \Big) \leq \sum\limits_{i=0}^{e-1} a_i \Big( \left|p^{i} G \right| - \left|p^{i+1} G \right| \Big). $$
But this inequality is equivalent to 
$$ 0 \leq a_{e-1} \Big(\grindex{G}{G_{\mf{p}}} -1  \Big) + \sum\limits_{i=1}^{e-1} (a_i - a_{i-1}) \Big( \big|p^iG| - \grindex{G}{G_{\mf{p}}} \big| p^iG_{\mf{p}} \big| \Big).   $$
Note that $p^e$ is the exponent of $G$. The first summand on the right hand side is clearly nonnegative, and both of parentheses occuring in the sum are nonpositive.
\end{proof}

\begin{cor}\label{FuehrerDiskriminantenAbschaetzung} For a $G$-extension $E/F$, we obtain 
$$ \Norm{\mf{d}(E/F)} \leq \Norm{\mf{f}(E/F)}^{\mG} \Norm{\tilde{\mf{f}}(E/F)}^{|pG|-1}, $$
where $\mG=\sum\nolimits_{i=0}^{e-1} p^{-i} (|p^iG|-|p^{i+1}G|)$ and $\tilde{\mf{f}}(E/F)$ is the squarefree part of $\mf{f}(E/F)$. 
\end{cor}
\begin{proof} This follows from Lemma \ref{FuehrerDiskriminantenFormelAbschaetzung} and the obvious inequality $\lceil x \rceil \leq  x  + 1$.
\end{proof}

The inequations above provide a sharp estimation of the discriminant in terms of the conductor and its support. An easy example for the sharpness is given by the $\cyclic{4}$-extension of the local field $K=\ffield{2}((t))$ with norm group $\langle t, 1 + t^3 \ffield{2}[[t]] \rangle \leq K^{\times}$. It has conductor exponent~$3$ and discriminant exponent~$8$.

\begin{theorem} Let $F$ be a global function field of characteristic $p$ and $G$ be a finite abelian $p$-group. Then the number of Artin-Schreier-Witt extensions $E/F$ with Galois group $G$ has asymptotic\footnote{We write $g(X)\in\grossOmega{f(X)}$ for $f(X)\in\grossO{g(X)}$. Note that we make use of a slightly different definition of the Landau symbols, as introduced in the beginning of this article.}
$$ Z(F,G;X) \in \grossOmega{X^{a_p(G)(1-\delta)}} $$
for some $\delta$ with $0\leq \delta \leq (|pG|-1)/(|G|-1)$. 
\end{theorem}
\begin{proof} Let $\mmG = \mG + |pG|-1$. Then Corollary~\ref{FuehrerDiskriminantenAbschaetzung} provides the rough inequality $\Norm{\mf{d}(E/F)} \leq \Norm{\mf{f}(E/F)}^{\mmG}$ by estimating the squarefree part by the conductor itself\footnote{This inequation is rough indeed. Note that every ramified prime is wildly ramified in $p$-extensions and hence has conductor exponent~$2$ at least.}. We hence obtain
$$ Z(F,G;X) \geq C(F,G;X^{1/\mmG}) \in \grossOmega{X^{\alpha_p(G)/\mmG}} $$
by Theorem~\ref{main2}. Since we have $a_p(G) = \alpha_p(G)/\mG$, the $X$-exponent of the lower bound is identical to 
$$ \frac{\alpha_p(G)}{\mmG}  = a_p(G) ( 1 - \delta ) $$ 
with 
$$ \delta = \frac{\mmG-\mG}{\mmG} = \frac{|pG|-1}{|pG|-1 + \sum\nolimits_{i=0}^{e-1} p^{-i} (|p^iG|-|p^{i+1}G|)} = \frac{|pG|-1}{|G|-1 + \sum\nolimits_{i=1}^{e-1} p^{-i}(|p^iG|-|p^{i+1}G|)} . $$ 
This proves the proposed bound for $Z(F,G;X)$.
\end{proof}

The above estimation for the defect $\delta$ has some individual potential for improvement, depending on the given group, but it provides the most handy one for all groups of interest. In particular, it confirms the lower bound for elementary abelian $p$-groups given by Theorem 2.2 in \cite{Lagemann2012}, since we have $\delta = 0$ in this case. In the same matter as above, we can give an upper bound for $Z(F,G;X)$, too. Let 
$$ d_p(G) = \frac{\alpha_p(G)}{|G| (1 - p^{-1})}. $$

\begin{theorem} Let $F$ be a global function field of characteristic $p$ and $G$ be a finite abelian $p$-group. Then the number of Artin-Schreier-Witt extensions $E/F$ with Galois group $G$ has asymptotic
$$ Z(F,G;X) \in \grossO{X^{d_p(G)}\log(X)^{\beta(F,G)-1}}. $$
\end{theorem}
\begin{proof} Let $E/F$ be a $G$-extension with conductor $\mf{f}$ and discriminant $\mf{d}$, and let $\chi : \Gal(E/F) \rightarrow \QQ/\ZZ$ be a character of $\Gal(E/F)$ of order $p$. Every character $\lambda \in \hat{G}$ being twisted with $\chi$ has at least conductor $\mf{f}(\chi)$. By the F\"uhrerdiskriminantenformel, we hence obtain $\mf{f}^m \mid \mf{d}$ with $m = |G|(1-p^{-1})$, where equality corresponds to the case, in which $\chi$ has conductor $\mf{f}(\chi) = \mf{f}$ and all of the $|G|/|\langle \chi \rangle|$ characters not being twisted with $\chi$ are unramified. Hence, every $G$-extension with discriminant norm $X$ has conductor norm $X^{1/m}$ at most. This results in 
$$ Z(F,G;X) \leq C(F,G;X^{1/m}) \in \grossO{X^{d_p(G)}\log(X)^{\beta(F,G)-1}}, $$
by Theorem \ref{main2}. 
\end{proof} 

Admittedly, the proof for the upper bound describes a very unrealistic scenario, but at least it yields a considerable better bound than $C(F,G;X)$. The situation considered for the lower bound is far more likely, but I suppose that the applied estimation is still pessimistic and we might omit the defect $\delta$ such that $a_p(G)$ gives the $X$-exponent of $Z(F,G;X)$. 

\begin{conjecture} Let $F$ be a global function field of characteristic $p$ and $G$ be a finite noncyclic abelian $p$-group. Then I conjecture that the number of Artin-Schreier-Witt extensions $E/F$ with Galois group $G$ has asymptotic
$$ Z(F,G;X) \sim c(F,G) X^{a_p(G)} \log(X)^{\beta(F,G)-1}. $$ 
\end{conjecture}

This conjecture has an interesting comparison to the results of David Wright \cite{Wright1989}. He determined the asymptotics of $Z(F,G;X)$ with the exception of finite abelian $p$-groups in positive characteristic $p$. In the left-out case, he only conjectured that the asymptotics should behave analogously. His proposed $X$-exponent $a(G)$ coincides with $a_p(G)$ for the simple cyclic $p$-group $\cyclic{p}$ and is even larger than $a_p(G)$ for nonsimple cyclic $p$-groups. But $a_p(G)$ breaks $a(G)$ for noncyclic abelian $p$-groups with the only exception of the Klein four group $\cyclic{2}^2$. With regard to the methods of embedding problems explored by J\"urgen Kl\"uners and Gunter Malle in \cite{KluenersMalle2004}, $a(G)$ might be the right $X$-exponent for cyclic groups. Hence the comparison of $a_p(G)$ and $a(G)$ shows, that things may vary for cyclic and noncyclic Artin-Schreier-Witt extensions. A hint for the correctness of my conjecture might be that the local discriminant asymptotics for $G$ have $X$-
exponent\footnote{See Satz 2.1 in \cite{LagemannAsymptotik}.}
$$ \frac{(p-1)\sum\nolimits_{i=1}^e p^{e-i}r_i(G) }{p\sum\nolimits_{i=0}^{e-1} p^{e-i}(|p^iG|-|p^{i+1}G|)}. $$
If we consider the local Dirichlet series $\Psi(F_{\mf{p}},G;s)$ generated by the discriminants of the local $G$-extensions $E/F_{\mf{p}}$, then its Euler product $\Gamma(s) = \prod\nolimits_{\mf{p}} \Psi(F_{\mf{p}},G;s)$ occurs as partial sum of $\Psi(F,G;s)$ generated by the discriminants of the global $G$-extensions $E/F$.\footnote{See sections 3-4 in \cite{Wright1989}.} The local series $\Psi(F_{\mf{p}},G;s)$ are infinite series, but they can be meromorphic continued as rational functions in $\N{p}^{-s}$, whence $\Psi(F_{\mf{p}},G;s)$ can described by finitely many expressions in $\N{p}$. The Euler product $\Gamma(s)$ converges if and only if every $\N{p}$-exponent of $\Psi(F_{\mf{p}},G;s)$ is less than $-1$. A crucial $\N{p}$-exponent occurring in $\Psi(F_{\mf{p}},G;s)$ is 
$$ \sum\limits_{i=1}^e p^{e-i}r_i(G) - \sum\limits_{i=1}^e p^{e-i}\Big(\big|p^iG\big|-\big|p^{i+1}G\big|\Big) s.   $$
This exponent is less than $-1$ for $\Re(s) > a_p(G)$. Admittedly, this is a very vague explanation, but it describes an analytic local-global principle and provides a prospect for an analytical attack on the conjecture. Unfortunately, the analysis is very hard and tedious, and there are some side-effects not being described above (e.g. $\Gamma(s)$ not being the only summand of $\Psi(F,G;s)$), which do not allow a straight-forward proof of this conjecture.

\section{Notations and preliminaries}

Let $F$ denote a global function field with constant field $\ffield{q}$ of cardinality $q$, algebraically closed in $F$. For a prime $\mf{p}$ of $F$, let $\ffield{\mf{p}}\geq \ffield{q}$ denote the residue class field, and $\N{p} = |\ffield{\mf{p}}|$ its norm. The valuation of $\mf{p}$ is denoted by $\val{p}$.  Let $g_F$ be the genus, $h_F$ the class number, and $\Cl$ the class group of $F$. In particular, we have $\Cl\iso\ZZ \times \Cl[h_F]$. For a module $\mf{m}$, an effective divisor of $F$ that is, let $F^{\mf{m}}$ be the multiplicative group of functions $x\in F^{\times}$ with divisor coprime to $\mf{m}$, and let $F_{\mf{m}}\leq F^{\mf{m}}$ be the ray mod $\mf{m}$ of functions $x\in F^{\mf{m}}$ with $x\in\open{p}{\ord{p}{\mf{m}}}$ for all $\mf{p}\mid\mf{m}$.  
The ray class group $\Cl_{\mf{m}}$  mod $\mf{m}$ is defined by the exact sequence of abelian groups
\begin{equation}\label{RayClassGroupSequence}
  1 \rightarrow \ffield{q}^{\times} \rightarrow F^{\mf{m}}/F_{\mf{m}} \rightarrow \Cl_{\mf{m}} \rightarrow \Cl \rightarrow 1. 
\end{equation}
By class field theory, there is a bijective map of finite abelian extensions $E/F$ with module $\mf{m}$ and subgroups $U\leq\Cl_{\mf{m}}$ of finite index, such that $\Gal(E/F) \iso \Cl_{\mf{m}}/U$.\footnote{E.g. see p. 139 and Theorem 9.23, p. 140 in \cite{RosenNTFF} for the sequence (\ref{RayClassGroupSequence}) and the Artin reciprocity law respectively.} Let $U_{\mf{m}}$ be the following finite $p$-Sylow group\footnote{E.g. see section 9.2 in \cite{Hess2003} for the isomorphisms.} 
\begin{equation}\label{EinseinheitenGroup} 
U_{\mf{m}} = (F^{\mf{m}}/F_{\mf{m}})[p^{\infty}] \iso \prod\limits_{\mf{p}^m\mid\mid\mf{m}} U_{\mf{p}^m} \iso \prod\limits_{\mf{p}^m\mid\mid\mf{m}} \langle 1 + \mf{p}\rangle /\langle 1+ \mf{p}^m \rangle,
\end{equation}
where the product is to read as follows. The relation $\mf{p}^m\mid\mid\mf{m}$ stands for $\ord{p}{\mf{m}}=m > 0$. For an expression $f_{\mf{p}}(m)$, the product $\prod\nolimits_{\mf{p}^m\mid\mid\mf{m}} f_{\mf{p}}(m)$ stands for $\prod\nolimits_{\mf{p}\mid\mf{m}} f_{\mf{p}}(\val{p}(\mf{m}))$. For example, let $f_{\mf{p}}(m)=\mf{p}^{m-1}$ and we have $\prod\nolimits_{\mf{p}^m\mid\mid\mf{m}} \mf{p}^{m-1} = \prod\nolimits_{\mf{p}\mid\mf{m}} \mf{p}^{\val{p}(\mf{m})-1}$, which is still an effective divisor. The same definition holds for sums. For example, the divisor above has degree $\sum\nolimits_{\mf{p}^m\mid\mid\mf{m}} (m-1)\deg\mf{p} = \sum\nolimits_{\mf{p}\mid\mf{m}} (\ord{p}{\mf{m}}-1)\deg \mf{p} \geq 0$. For the trivial module $\mf{m}=\mf{1}$, these products and sums are particularly defined as empty ones. Further, let $S_i = \big\{ x\in F^{\times} : (x)=\mf{a}^{p^i} \textup{ for some divisor } \mf{a} \big\}/F^{p^i}$ be the $p^i$-Selmer group, and let $S_{i,\mf{m}}= \big\{ [x]\in S_i : x\in F_{\mf{m}} \big\}$ be the Selmer 
ray group mod $\mf{m}$ of $F$. All these notations coincide with those used in \cite{Lagemann2012} except of the definition of $U_{\mf{m}}$. In addition, lets introduce the notation of the ($p$-)signature $\tupel{r}(A)= (r_i(A))_{i\geq 1}$ for finitely generated abelian groups $A$ with  $r_i(A) = \log_p \grindex{p^{i-1}A}{p^iA}$. In other words, $\tupel{r}(A)$ is a stationary monotone decreasing sequence of nonnegative integers ending in the free rank of $A$. Let $e\geq 1$ be minimal with $r_{e+i}(A) = r_{e+1}(A)$ for all $i\geq 1$. Then we obtain the isomorphism class of $A$ via
$$ A \iso \ZZ^{r_{e+1}(A)} \times \prod\limits_{i=1}^e \cyclic{p^i}^{\, r_i(A)-r_{i+1}(A)}. $$

\section{Hasse formula}

In this section, we shall determine the $p$-signature of the ray class group $\Cl_{\mf{m}}$. 

\begin{prop}\label{ExSeqA} We have the exact sequence of finite abelian $p$-groups 
$$ 1 \rightarrow S_{i,\mf{m}} \rightarrow \Cl[p^i] \rightarrow U_{\mf{m}}/U_{\mf{m}}^{\, p^i} \rightarrow \Cl_{\mf{m}}/\Cl_{\mf{m}}^{\, p^i} \rightarrow \Cl/\Cl^{\, p^i} \rightarrow 1. $$
\end{prop}
\begin{proof} Similar to the proof of Proposition 4.1 in \cite{Lagemann2012}, we obtain Proposition \ref{ExSeqA} by tensoring sequence (\ref{RayClassGroupSequence}) with $\cyclic{p^i}$.
\end{proof}

\begin{prop}\label{ExSeqB} We have the exact sequence of finite abelian $p$-groups
$$ 1 \rightarrow S_{i-1,\mf{m}} \rightarrow S_{i,\mf{m}} \rightarrow S_{1,\mf{m}}. $$ 
\end{prop}
\begin{proof} Since the Frobenius homomorphism is injective, we find $S_{i-1,\mf{m}}$ embedded in $S_{i,\mf{m}}$ via $[x] \mapsto [x^p]$. Further, $S_{i,\mf{m}}$ naturally maps to $S_{1,\mf{m}}$ via $[x] \mapsto [x]$. The kernel of the latter map is generated by Selmer classes $[x^p]\in S_{i,\mf{m}}$ with divisor of the form $(x^p)=\mf{a}^{p^{i}}$, which implies $(x)=\mf{a}^{p^{i-1}}$ and $[x]\in S_{i-1,\mf{m}}$.
\end{proof}

With Proposition \ref{ExSeqB}, we can identify $S_{i-1,\mf{m}}$ with its isomorphic image in $S_{i,\mf{m}}$, such that the index $\grindex{S_{i,\mf{m}}}{S_{i-1,\mf{m}}}$ is well-defined.

\begin{lemma}\label{piRank} For any integer $m\geq 1$, let $w_i(m) = \lfloor (m-1)/p^{i-1} \rfloor - \lfloor (m-1)/p^i \rfloor$, and $w_i(0) = 0$. Further, let $\delta_{i,1} \in \big\{ 0,1 \big\}$ be the usual Kronecker symbol. Then we have
$$ \grindex{\Cl_{\mf{m}}^{p^{i-1}}}{\Cl_{\mf{m}}^{p^i}} = p^{\delta_{i,1}} \grindex{S_{i,\mf{m}}}{S_{i-1,\mf{m}}} \grindex{U_{\mf{m}}^{p^{i-1}}}{U_{\mf{m}}^{p^i}} = p^{\delta_{i,1}} \grindex{S_{i,\mf{m}}}{S_{i-1,\mf{m}}} \prod\limits_{\mf{p}^m \mid\mid \mf{m}} \N{p}^{w_i(m)}. $$ 
\end{lemma}
\begin{proof} With the exact sequence of Proposition \ref{ExSeqA}, we obtain
$$ \grindex{\Cl_{\mf{m}}}{\Cl_{\mf{m}}^{p^i}} = \frac{|S_{i,\mf{m}}| \grindex{U_{\mf{m}}}{U_{\mf{m}}^{p^i}} \grindex{\Cl}{\Cl^{p^i}}}{|\Cl[p^i]|} = p\ |S_{i,\mf{m}}| \grindex{U_{\mf{m}}}{U_{\mf{m}}^{p^i}}.  $$
For $i=1$, this yields the first proposed equality. For $i\geq 2$ however, it follows from
$$ \grindex{\Cl_{\mf{m}}^{p^{i-1}}}{\Cl_{\mf{m}}^{p^i}} = \frac{\grindex{\Cl_{\mf{m}}}{\Cl_{\mf{m}}^{p^i}}}{\grindex{\Cl_{\mf{m}}}{\Cl_{\mf{m}}^{p^{i-1}}}} = \frac{|S_{i,\mf{m}}| \grindex{U_{\mf{m}}}{U_{\mf{m}}^{p^i}}}{|S_{i-1,\mf{m}}| \grindex{U_{\mf{m}}}{U_{\mf{m}}^{p^{i-1}}}} = \grindex{S_{i,\mf{m}}}{S_{i-1,\mf{m}}} \grindex{U_{\mf{m}}^{p^{i-1}}}{U_{\mf{m}}^{p^i}}. $$ 
Then the second equality follows from the isomorphism in (\ref{EinseinheitenGroup}) and the well-known Einseinheitensatz\footnote{E.g. see Proposition 6.2, p.18 in \cite{Fesenko2002}.}. 
\end{proof}

\begin{cor}\label{SquareFreeParts} Let $\mf{n}$ be a squarefree module coprime to $\mf{m}$. Then we have $\grindex{S_{i,\mf{mn}}}{S_{i-1,\mf{mn}}} = \grindex{S_{i,\mf{m}}}{S_{i-1,\mf{m}}}$. 
\end{cor}
\begin{proof} Ramified abelian $p$-extensions of $F$ are automatically wildly ramified, whence their local conductors can not be squarefree. We therefore find the Einseinheiten groups $U_{\mf{mn}}$ and $U_{\mf{m}}$ as well as the ray class factor groups $\Cl_{\mf{mn}}/\Cl_{\mf{mn}}^{p^i}$ and $\Cl_{\mf{m}}/\Cl_{\mf{m}}^{p^i}$ to be isomorphic. Hence, the proposed equality follows from Lemma \ref{piRank}.
\end{proof}

\begin{lemma}\label{SelmerIndex} Let $\mf{m}$ be a module with the condition 
$$ \sum\limits_{\mf{p}^m \mid\mid \mf{m}} (m-1)\deg \mf{p} > 2g_F -2. $$
Then we have $\grindex{S_{i,\mf{m}}}{S_{i-1,\mf{m}}} = 1$. 
\end{lemma}
\begin{proof} By Lemma 4.4 in \cite{Lagemann2012}, the Selmer ray group $S_{1,\mf{m}}$ is trivial for a module $\mf{m}$ with the given condition. We hence obtain $\grindex{S_{i,\mf{m}}}{S_{i-1,\mf{m}}} = 1$ 
by Proposition \ref{ExSeqB}. 
\end{proof}

\section{M\"obius formula}

With the results of the last section, we can count the Artin-Schreier-Witt extensions with given module. To get access on the number of extensions with given conductor, we shall use the M\"obius inversion formula. Beforehand, we define multiple functions on modules to make the formula more comfty. Let $\tupel{s}(\mf{m}) = (s_i(\mf{m}))_{i\geq 1}$ with 
$$s_i(\mf{m}) = \grindex{S_{i,\mf{m}}}{S_{i-1,\mf{m}}},$$ 
such as $\tupel{u}(\mf{m}) = (u_i(\mf{m}))_{i\geq 1}$ with 
$$u_i(\mf{m}) = \grindex{U_{\mf{m}}^{p^{i-1}}}{U_{\mf{m}}^{p^i}} = \prod\limits_{\mf{p}^m\mid\mid\mf{m}} u_i(\mf{p}^m) = \prod\limits_{\mf{p}^m\mid\mid\mf{m}} \N{p}^{w_i(m)},$$ 
and $\tupel{x}(\mf{m}) = (x_i(\mf{m}))_{i\geq 1} $ with
$$x_i(\mf{m}) = s_i(\mf{m}) u_i(\mf{m}) = p^{-\delta_{i,1}} \grindex{\Cl_{\mf{m}}^{p^{i-1}}}{\Cl_{\mf{m}}^{p^i}} .$$ 
From now on, we fix the abelian $p$-group $G$. For any subgroup $H \unlhd G$, let 
$$ s_H(\mf{m}) = \prod\limits_{i\geq 1} s_i(\mf{m})^{r_i(H)}, \qquad u_H(\mf{m}) = \prod\limits_{i\geq 1} u_i(\mf{m})^{r_i(H)}, \qquad x_H(\mf{m}) = \prod\limits_{i\geq 1} x_i(\mf{m})^{r_i(H)}, $$
and
$$ c_{H}(\mf{m}) = \sum\limits_{\mf{n}\mid\mf{m}} \mu(\mf{mn}^{-1}) x_H(\mf{n}), $$ 
where $\mu(\mf{n})$ denotes the M\"obius function for modules. These are well-defined integer-valued functions on modules.

\begin{prop}\label{MoebiusFormula} For any module $\mf{m}$, let $c_{\mf{m}}$ be the number of $G$-extensions $E/F$ with conductor $\mf{m}$. Further, let 
$$ e(\tupel{X}) = \frac{1}{|\Aut(G)|} \sum\limits_{pG \unlhd H \unlhd G} \mu(G/H) \grindex{H}{pH} \tupel{X}^{\tupel{r}(H)} = \sum\limits_{H \unlhd G} e_{H} \tupel{X}^{\tupel{r}(H)}, $$
where $\mu(A)$ denotes the Delsarte-M\"obius function for finite abelian groups\footnote{E.g. see Appendix A.}. Then we have
$$ c_{\mf{m}} = \sum\limits_{H \unlhd G} e_{H} c_{H}(\mf{m}). $$
\end{prop}
\begin{proof} Let $f_G(\tupel{X}) = \sum\nolimits_{pG \unlhd H \unlhd G} \mu(G/H) \tupel{X}^{\tupel{r}(H)}$ be the Delsarte-M\"obius polynomial of $G$ as defined in Lemma \ref{DelsarteMoebiusLemma}. For any finitely generated abelian group $A$, we particularly have that $f_G(\tupel{x})/|\Aut(G)|$ gives the number $\kappa(A,G)$ of subgroups $B\unlhd A$ with quotient $G$, where $\tupel{x} = (\grindex{p^{i-1}A}{p^iA})_{i\geq 1}$. If we observe  $\grindex{H}{pH} = p^{r_1(H)}$, we see
$$ e(\tupel{X}) = \frac{f_G\big((p^{\delta_{i,1}}X_i)_{i\geq 1}\big)}{|\Aut(G)|}. $$
Together with Lemma \ref{piRank} and class field theory, we now see that $ e(\tupel{x}(\mf{m})) = \kappa(\Cl_{\mf{m}},G) $
gives the number of $G$-extensions $E/F$ with module $\mf{m}$. On the other hand, the number of $G$-extensions $E/F$ with module $\mf{m}$ sums up to
$$ e(\tupel{x}(\mf{m})) = \sum\limits_{H \unlhd G} e_Hx_H(\mf{m}) = \sum\limits_{\mf{n}\mid\mf{m}} c_{\mf{n}}. $$
This yields 
$$ c_{\mf{m}} = \sum\limits_{\mf{n}\mid\mf{m}} \mu(\mf{mn}^{-1}) e(\tupel{x}(\mf{n})) = \sum\limits_{H \unlhd G} e_{H} c_{H}(\mf{m}) $$
by the inversion formula of M\"obius.
\end{proof}

The M\"obius formula above nearly gives a local-global principle for $c_{\mf{m}}$. For example, suppose that $F$ has a class number $h_F$ indivisible by $p$. Then its Selmer groups $S_i \iso \Cl[p^i]$ are trivial, whence all Selmer indices $s_i(\mf{m})$ are trivial. Thus the number $c_{\mf{m}}$ of $G$-extensions $E/F$ with conductor $\mf{m}$ equals
$$ b_{\mf{m}} = \sum\limits_{H \unlhd G} e_{H} b_{H}(\mf{m}) $$
with
$$ b_{H}(\mf{m}) = \sum\limits_{\mf{n}\mid\mf{m}} \mu(\mf{mn}^{-1}) u_H(\mf{n}) = \sum\limits_{\mf{n}\mid\mf{m}} \mu(\mf{mn}^{-1}) \prod\limits_{i\geq 1} u_i(\mf{n})^{r_i(H)}. $$
We find $u_H(\mf{m})$ to be multiplicative, and so is $b_{H}(\mf{m})$.
 This yields
\begin{equation}\label{armultiplicative}
 b_{H}(\mf{m}) = \prod\limits_{\mf{p}^m\mid\mid\mf{m}} b_{H}(\mf{p}^m) = \prod\limits_{\mf{p}^m\mid\mid\mf{m}} \Big( u_H(\mf{p}^m) - u_H(\mf{p}^{m-1}) \Big).  
\end{equation}
The last equation follows from the fact that the M\"obius function is only nonzero for squarefree modules. For general global function fields $F$, we shall see that a huge family of modules $\mf{m}$ allows the formula $c_{\mf{m}} = b_{\mf{m}}$, whereas formulas for the infinite complementary family depend only on finitely many modules. This yields a crucial tool for disecting $\Phi(F,G;s)$ in a finite sum of Euler products, as we shall see in the next section. We call a module $\mf{m}$ squareful, if every prime divisor $\mf{p}\mid\mf{m}$ has at least order $\val{p}(\mf{m})\geq 2$. Let $\mc{M}$ be the finite set of squareful modules $\mf{m}$, supported only by primes $\mf{p}$ with $\deg\mf{p} \leq 2g_F-2$ and $\ord{p}{\mf{m}}\leq 2g_F$. Note that $\mc{M}$ only consists of the trivial module $\mf{m}=\mf{1}$ in case of $g_F=0,1$.

\begin{lemma}\label{MoebiusDifferent} Let $b_{\mf{m}} = \sum\nolimits_{H \unlhd G} e_{H} b_{H}(\mf{m})$, 
and let $a_{\mf{m}} = c_{\mf{m}} - b_{\mf{m}}$. Then $a_{\mf{m}}$ fulfils the following formulas. 
\begin{enumerate}[(a)]
 \item Let $\mf{m}=\mf{m}_0\mf{m}_1^2$ with $\mf{m}_0\in\mc{M}$ and $\mf{m}_1$ being a squarefree module supported only by primes $\mf{p}$ of degree $\deg \mf{p} > 2g_F-2$. Then we have
$$ a_{\mf{m}} = \mu(\mf{m}_1) a_{\mf{m}_0}.  $$
 \item Let $\mf{m}$ be a module not covered by \textup{(a)}. Then we have
$$ a_{\mf{m}}  = 0. $$
\end{enumerate}
\end{lemma}
\begin{proof} For this proof, we may consider squareful modules $\mathfrak{m}$ only, since $a_{\mathfrak{m}}=0$ follows from $b_{\mathfrak{m}}=0$ and $c_{\mathfrak{m}}=0$ otherwise. The former follows from definition and the latter is based upon the fact that any ramified prime is wildly ramified in $p$-extensions. Therefore we assume $\mathfrak{m}$ to be squareful. By Proposition \ref{MoebiusFormula}, we have
\begin{equation}\label{FormulaAm} a_{\mf{m}} = \sum\limits_{H \unlhd G} e_{H}(c_{H}(\mf{m}) - b_{H}(\mf{m})) = \sum\limits_{H \unlhd G} e_{H} \sum\limits_{\mf{n}\mid\mf{m}} \mu(\mf{mn}^{-1}) (s_H(\mf{n}) - 1) u_H(\mf{n}).
\end{equation}
We first consider $\mf{m}$ not being covered by (a). Then there is a multiple prime divisor $\mf{p}^m \mid\mid \mf{m}$ with $m>2$, such that $\deg \mf{p} > 2g_F - 2$ or $m > 2g_F$ holds. Any divisor $\mf{n} \mid \mf{m}$ with $\mu(\mf{mn}^{-1}) \neq 0$ is divisible by $\mf{p}^{m-1}$, whence $s_H(\mf{n}) = 1$ follows from Lemma \ref{SelmerIndex}. This yields $a_{\mf{m}} = 0$ by formula (\ref{FormulaAm}), as proposed. Now, let $\mf{m} = \mf{m}_0\mf{m}_1^2$ be a module as considered in (a). Any divisor $\mf{n}\mid\mf{m}$ with $\mu(\mf{m}\mf{n}^{-1})\neq 0$ has necessarily the form $\mf{n}=\mf{n}_0\mf{m}_1\mf{n}_1$ with $\mf{n}_0\mid\mf{m}_0$ and $\mf{n}_1\mid\mf{m}_1$, as $\mf{m}\mf{n}^{-1}$ would contain a square otherwise. Considering this form of $\mf{n}$, we have $s_H(\mf{n}) = 1$ for $\mf{n}_1 \neq \mf{1}$ by Lemma \ref{SelmerIndex}. Hence, the right-hand side of formula (\ref{FormulaAm}) only depends 
on modules of the form $\mf{n} = \mf{n}_0\mf{m}_1$ with $\mf{n}_0 \mid \mf{m}_0$, which leads to
$$ a_{\mf{m}} = \mu(\mf{m}_1) \sum\limits_{H \unlhd G} e_{H} \sum\limits_{\mf{n}_0\mid\mf{m}_0} \mu(\mf{m}_0\mf{n}_0^{-1}) (s_H(\mf{n}_0\mf{m}_1) - 1) u_H(\mf{n}_0\mf{m}_1). $$
We obtain $s_H(\mf{n}_0\mf{m}_1) = s_H(\mf{n}_0)$ by Corollary \ref{SquareFreeParts}, such as $u_H(\mf{n}_0\mf{m}_1) = u_H(\mf{n}_0)$ by Lemma \ref{piRank}, which yields the proposed identity $a_{\mf{m}} = \mu(\mf{m}_1) a_{\mf{m}_0}$.
\end{proof}

\section{Euler formula}

This section contains a rather complete description of the Dirichlet series
$$ \Phi(F,G;s) = \sum\limits_{\Gal(E/F) \iso G} \Norm{\mf{f}(E/F)}^{-s} = \sum\limits_{\mf{m}} c_{\mf{m}} \N{m}^{-s}, $$  
generated by the $G$-extensions $E/F$ and by the modules $\mf{m}$ of $F$ respectively. By Lemma \ref{MoebiusDifferent}, its coefficients have the decomposition
\begin{equation}\label{cmdecomp} c_{\mf{m}} = a_{\mf{m}} + b_{\mf{m}} = a_{\mf{m}} + \sum\limits_{H \unlhd G} e_Hb_H(\mf{m}) = a_{\mf{m}} + \sum\limits_{H \unlhd G} e_H \prod\limits_{\mf{p}^m\mid\mid\mf{m}} b_H(\mf{p}^m). 
\end{equation}
This yields a decomposition of $\Phi(F,G;s)$ in a finite sum of Euler products $\Phi_H(s)$ and an error term $\Upsilon_G(s)$ with
\begin{equation}\label{PhiH}
\Phi_H(s) = \sum\limits_{\mf{m}} b_H(\mf{m}) \N{m}^{-s} = \prod\limits_{\mf{p}}  \sum\limits_{m\geq 0} b_{H}(\mf{p}^m) \N{p}^{-ms} .  
\end{equation}
and 
\begin{equation}\label{UpsilonG}
\Upsilon_G(s) = \sum\limits_{\mf{m}} a_{\mf{m}}\N{m}^{-s}.
\end{equation}

\begin{lemma}\label{EulerFormula} The Dirichlet series $\Phi(F,G;s) = \sum\nolimits_{\mf{m}} c_{\mf{m}} \N{m}^{-s}$ has the decomposition
$$ \Phi(F,G;s) = \sum\limits_{H \unlhd G} e_{H} \Phi_{H}(s) + \Upsilon_G(s)  $$
in the domain of convergence.  
\end{lemma}
\begin{proof} This follows from the identities (\ref{cmdecomp}), (\ref{PhiH}), and (\ref{UpsilonG}).
\end{proof}

\begin{prop}\label{Upsilon} The Dirichlet series $\Upsilon_G(s) = \sum\nolimits_{\mf{m}} a_{\mf{m}} \N{m}^{-s}$ is holomorphic for $\Re(s) > 1/4$ and meromorphic continuable for all complex arguments $s$.  
\end{prop}
\begin{proof} By Lemma \ref{MoebiusDifferent}, the modules $\mf{m}$ with $a_{\mf{m}} \neq 0$ are of the form $\mf{m} = \mf{m}_0\mf{m}_1^2$ with $\mf{m}_0\in\mc{M}$ and $\mf{m}_1$ being a squarefree module supported by primes of degree $\deg \mf{p} > 2g_F - 2$ only. This yields
$$ \Upsilon_G(s)  = \sum\limits_{\mf{m}_0\in\mc{M}} \sum\limits_{\mf{m}_1} \mu(\mf{m}_1) a_{\mf{m}_0} \Norm{\mf{m}_0\mf{m}_1^2}^{-s} = \left( \sum\limits_{\mf{m}_0\in\mc{M}} a_{\mf{m}_0}  \Norm{\mf{m}_0}^{-s} \right) \prod\limits_{\deg \mf{p} > 2g_F -2} \left(1-\N{p}^{-2s}\right) .$$
Since $\mc{M}$ is finite, the first sum provides an entire function. The restricted Euler product however is identical with 
$$ \prod\limits_{\deg \mf{p} > 2g_F -2} \left(1-\N{p}^{-2s}\right) = \zeta_F(2s)^{-1} \prod\limits_{\deg \mf{p} \leq 2g_F -2}  \left(1-\N{p}^{-2s}\right)^{-1} $$
and thus has poles on the imaginary axes $\Re(s)=0$ and $\Re(s)=1/4$, by the theorem of Hasse-Weil\footnote{E.g. see Theorem 5.2.1 and Remark 5.2.2, p. 197 in \cite{StiALF}.}. 
\end{proof}

Since the convergence abscissa of $\Phi(F,G;s)$ is claimed to be at least $\alpha_p(G) \geq 1$, the error term $\Upsilon_G(s)$ is quite insignificant. And even more, we can compute $\Upsilon_G(s)$ with the knowledge of only finitely many values of $c_{\mf{m}}$, namely those for $\mf{m}\in\mc{M}$. This error term is hence quite manageable, but it depends strongly on the given datas $F$ and $G$. For the Euler products $\Phi_{H}(s)$ however, there is only a weak dependence on $F$ and $G$, as $\Phi_{H}(s)$ depends only on the primes of $F$ and the signature of $H$. In particular, $\Phi_{H}(s)$ depends on the isomorphism class of $H$ only. Thus we just need to investigate the Euler products $\Phi_G(s)$ for any abelian $p$-group $G$, and aim to describe $\Phi_{G}(s)$ as a product of zeta functions.

\begin{prop}\label{EulerFactors} The Euler factors 
$$ \Phi_{G,\mf{p}}(s) = \sum\limits_{m\geq 0} b_H(\mf{p}^m)\N{p}^{-ms} $$ 
of $\Phi_G(s)$ are meromorphic continuable for all complex arguments $s$ with $\Phi_{G,\mf{p}}(s) = \Lambda_{G,\mf{p}}(s) \Psi_{G,\mf{p}}(s)$, where
$$ \Lambda_{G,\mf{p}}(s) = \prod\limits_{l=2}^{p^e} \zeta_{F,\mf{p}} \left( ls - \sum\limits_{i=1}^e r_i(G) w_i(l) \right) $$ 
and
$$ \Psi_{G,\mf{p}}(s) = \left( 1 + \sum\limits_{l=1}^{p^e-1} u_G(\mf{p}^l) \N{p}^{-ls} \right) \prod\limits_{l=1}^{p^e-1} \left( 1 - u_G(\mf{p}^l) \N{p}^{-ls} \right). $$ 
\end{prop}
\begin{proof} It is clear, that $\Lambda_{G,\mf{p}}(s)$ and $\Psi_{G,\mf{p}}(s)$ are meromorphic functions on the whole complex plane, whence we only need to verify the formal identity $\Phi_{G,\mf{p}}(s) = \Lambda_{G,\mf{p}}(s) \Psi_{G,\mf{p}}(s)$. To do so, we introduce two new notations for this proof and the proof of Proposition \ref{PsiG}. Let 
$$ U_m = u_G(\mf{p}^m) \N{p}^{-ms} = \N{p}^{-ms} \prod\limits_{i = 1}^e \N{p}^{r_i(G)w_i(m)},  $$
where the latter equality follows from Lemma \ref{piRank} and $r_i(G) = 0$ for $i>e$. Further, let 
$$ B_m = b_G(\mf{p}^m) \N{p}^{-ms} = \begin{cases} 1, & m = 0, \\ \N{p}^{-ms} \left(u_G(\mf{p}^m) - u_G(\mf{p}^{m-1}) \right) = U_m - U_{m-1}U_1, & m\geq 1, \end{cases} $$
where the case by case equalities follow easily from definition. A cruical identity shall be 
\begin{equation}\label{CruicalID} B_{kp^e + l} = U_{p^e}^k B_l
\end{equation}
for all integers $k \geq 0,l \geq 1$. To see equation (\ref{CruicalID}), recall $w_i(m) = \lfloor (m-1)/p^{i-1} \rfloor - \lfloor (m-1)/p^{i} \rfloor$ for $m \geq 1$. Then the claimed relation subsequently follows from $w_i(kp^e + l) = kw_i(p^e + 1) + w_i(l) = kw_i(p^e) + w_i(l)$ for $1 \leq i \leq e$. Now, we can start with verifying the proposed equation $\Phi_{G,\mf{p}} = \Lambda_{G,\mf{p}} \Psi_{G,\mf{p}}$. With the given notations and identity (\ref{CruicalID}), we have
$$ \Phi_{G,\mf{p}} = \sum\limits_{m\geq 0} B_m = 1 + \sum\limits_{k\geq 0} \sum\limits_{l=1}^{p^e} U_{p^e}^k B_l. $$
Thus $\Phi_{G,\mf{p}}$ has the meromorphic continuation
$$ \Phi_{G,\mf{p}} = 1 + \left( 1 - U_{p^e} \right)^{-1} \sum\limits_{l=1}^{p^e}B_l = \left( 1 - U_{p^e} \right)^{-1} \left( 1 - U_{p^e} + \sum\limits_{l=1}^{p^e}B_l \right).   $$
We hence obtain
$$  \Phi_{G,\mf{p}} = \left( 1 - U_{p^e} \right)^{-1} \left( 1 + \sum\limits_{l=1}^{p^e-1} U_l - \sum\limits_{l=1}^{p^e} U_{l-1}U_1\right) = \left( 1 - U_{p^e} \right)^{-1} (1 - U_1) \left( 1 + \sum\limits_{l=1}^{p^e-1} U_l \right). $$
This yields
$$ \Phi_{G,\mf{p}} \Psi_{G,\mf{p}}^{-1} =  \prod\limits_{l=2}^{p^e} (1-U_l)^{-1} = \Lambda_{G,\mf{p}}, $$
as proposed.
\end{proof}

\begin{prop}\label{EpsilonL}\label{FractionComparision} For any integer $2\leq l \leq p^e$, let 
$$ \varepsilon(l) = \alpha_p(G) - \frac{1 + \sum\nolimits_{i=1}^e r_i(G)w_i(l)}{l}, $$
and let $0 \leq f\leq e$ be the minimal integer such that $p^fG$ is cyclic. Then we have $\varepsilon(l)\geq 0$ with $\varepsilon(l)=0$ if and only if $p^f \mid l$. 
\end{prop}
\begin{proof} Recall $w_i(l) = \lfloor (l-1)/p^{i-1} \rfloor - \lfloor (l-1)/p^i \rfloor$. Since $w_i(p^e) = p^{e-i}(p-1)$ holds for $1\leq i \leq e$, we have $\varepsilon(p^e) = 0$. The inequation $\varepsilon(l)\geq 0$ is hence equivalent with the inequation
$$
 w(l) = l - p^e + \sum\limits_{i=1}^e r_i(G) \big( lw_i(p^e) - p^ew_i(l) \big) \geq 0
$$
for the nominator of $\varepsilon(l)=w(l)/(p^el)$, and we have $\varepsilon(l)=0$ if and only if $w(l)=0$ holds. We further set $v_i(l) = lw_i(p^e) - p^ew_i(l)$. Then $w(l)$ has the form
$$  w(l) =  \left(  l - p^e + \sum\limits_{i=1}^e v_i(l)  \right)
  + \left( (r_e(G)-1) \sum\limits_{i=1}^e v_i(l) \right)
 +  \sum\limits_{j=1}^{e-1} \left( (r_j(G)-r_{j+1}(G)) \sum\limits_{i=1}^j v_i(l) \right).
$$
Since $p^fG$ is cyclic, we have $r_i(G) = 1$ for $f+1 \leq i \leq e$. Thus we obtain
\begin{equation}\label{LHSClaim1}  w(l) =  \left(  l - p^e + \sum\limits_{i=1}^e v_i(l)  \right)
  + \left( (r_f(G)-1) \sum\limits_{i=1}^f v_i(l) \right)
 +  \sum\limits_{j=1}^{f-1} \left( (r_j(G)-r_{j+1}(G)) \sum\limits_{i=1}^j v_i(l) \right).
\end{equation}
Now we shall prove that the entries of the big parentheses are nonnegative respectively. Let $l-1=l_0+\ldots+l_{e-1}p^{e-1}$ be in $p$-adic representation, i.e. $0\leq l_k\leq p-1$ for $k = 0,\ldots,e-1$. By definition of $w_i(l)$, we have
$$ w_i(l) = \left(\sum\limits_{k=i}^{e-1} l_kp^{k-i}(p-1)\right) + l_{i-1},  $$
$$ lw_i(p^e) =  (p-1)p^{e-i} + \left(\sum\limits_{k=0}^{e-1} l_kp^k(p-1)p^{e-i}\right), $$
$$ p^ew_i(l) =  \left(\sum\limits_{k=i}^{e-1} l_kp^{k-i}(p-1)p^{e}\right) + l_{i-1}p^{e}, $$
and
$$ v_i(l) = (p-1) p^{e-i} + \left( \sum\limits_{k=0}^{i-1} l_kp^k (p-1)p^{e-i} \right)  - l_{i-1}p^{i-1}p^{e-i+1} $$
for $i=1,\ldots,e$. By summing from $i=1$ up to $j$ with $1\leq j\leq e$, we obtain
\begin{align*}\label{SumSi}
\sum\limits_{i=1}^j v_i(l)  =  & \Bigg(\sum\limits_{i=1}^j (p-1)p^{e-i}\Bigg) + \Bigg(\sum\limits_{k=0}^{j-1} l_kp^k \sum\limits_{i=k+1}^j (p-1)p^{e-i} \Bigg) - \Bigg(\sum\limits_{k=0}^{j-1} l_k p^k p^{e-k} \Bigg) \\
 = & \Bigg( p^e - p^{e-j} \Bigg) +  \sum\limits_{k=0}^{j-1} l_kp^k \Big( \big( p^{e-k} - p^{e-j}  \big) - p^{e-k} \Big) \\
 = & \ p^e - p^{e-j} \left( 1 + \sum\limits_{k=0}^{j-1} l_kp^k \right).
\end{align*}
Since the last parenthesis is at most $p^j$, this sum is nonnegative. Thus the two latter parentheses of (\ref{LHSClaim1}) are nonnegative, too. Note $r_i(G) \geq r_{i+1}(G)$ hereby. The first parenthesis of (\ref{LHSClaim1}) even vanishes due to
$$ l - p^e + \sum\limits_{i=1}^e v_i(l) = l - \left( 1 + \sum\limits_{k=0}^{e-1} l_kp^k \right) = 0.  $$
We hence have proven $w(l)\geq 0$, as desired. Further $w(l)$ is zero if and only if the latter two big parentheses of (\ref{LHSClaim1}) are also zero respectively. Firstly, we consider the case $f=0$, i.e. $G$ is cyclic. Then we have $r_i(G)=1$ for $1\leq i \leq e$, whence $w(l)=0$ follows for all integers $l$ in question, as proposed. In the case $f\neq 0$, we have $r_f(G) \neq 1$. Then the second parenthesis of (\ref{LHSClaim1}) is zero if and only if we have $l_0 = \ldots = l_{f-1} = p-1$, which is equivalent to $p^f \mid l$. In the same way, the third parenthesis is zero if and only if we have $r_j(G) = r_{j+1}(G)$ or $p^j \mid l$ for $j=1,\ldots,f$. Altogether, we have $w(l) = 0$ if and only if $p^f \mid l$. 
\end{proof}

\begin{prop}\label{PropertiesLambdaG} The function 
$$\Lambda_{G}(s) = \prod\limits_{l=2}^{p^e} \zeta_{F} \left( ls - \sum\limits_{i=1}^e r_i(G) w_i(l) \right) $$ 
has the following properties.
\begin{enumerate}
 \item[\textup{(a)}] Let $G$ be cyclic. Then $\alpha_p(G) = 1$ holds, and $\Lambda_{G}(s)$ is holomorphic for $\Re(s) > 1 - 1/p^e$ except for the set of poles 
$$ \left\{ 1 + \frac{2 \pi i}{\log(q)} \frac{j}{l} : 2 \leq l \leq p^e,\, 0 \leq j < l \right\} + \frac{2 \pi i}{\log(q)}\, \ZZ. $$
The periodic pole corresponding to $s=1$ is of maximal order $\beta(F,G) = p^e-1$. In case of $p^e \neq 2$, every other pole has less order.
 \item[\textup{(b)}] Let $G$ be noncyclic. Then $\Lambda_{G}(s)$ is holomorphic for $\Re(s) > \alpha_p(G) - \varepsilon$ with some $\varepsilon > 0$ except for the set of poles
$$ \left\{ \alpha_p(G) + \frac{2 \pi i}{\log(q)} \frac{j}{p^fl} : 1\leq l \leq p^{e-f},\, 0 \leq j < p^fl \right\} + \frac{2 \pi i}{\log(q)}\, \ZZ. $$
The periodic pole corresponding to $s=\alpha_p(G)$ is of maximal order $\beta(F,G) = p^{e-f}$. In case of $p^{e-f} \neq 1$, every other pole has less order. 
\end{enumerate}
\end{prop}
\begin{proof} Since the poles of $\zeta_F(s)$ are given by the values of $s$ with $q^s=1$ and $q^s=q$ respectively, the poles of $\Lambda_{G}(s)$ are given by the values of $s$ with $q^{ls} = \prod\nolimits_{i=1}^e q^{r_i(G)w_i(l)}$ and $q^{ls} = q\prod\nolimits_{i=1}^e q^{r_i(G)w_i(l)}$ for $l=2,\ldots,p^e$ respectively. Each of these poles has period $2\pi i/(\log(q)l)$. By Proposition \ref{FractionComparision}, the former poles are out of the region of interest due to 
$$ \alpha_p(G) - \frac{\sum\nolimits_{i=1}^e r_i(G)w_i(l)}{l} = \varepsilon(l) + \frac{1}{l} \geq \frac{1}{p^e}. $$ 
Furthermore, the latter poles meet $\alpha_p(G)$ in case of $p^f \mid l$ only. Otherwise they have positive distance to the convergence abscissa. This shows the existence of $\varepsilon = \min \big\{ \varepsilon(l) : 2\leq l \leq p^e, p^f \nmid l \big\}$ for assertion (b). The pole order of $a_p(G)$ is hence given by $\beta(F,G) = \left|\big\{ 2 \leq l \leq p^e : p^f \mid l \big\}\right|$. Since $2\pi i/(p^f\log(q))$ is the meet of all periods, only the poles $s=a_p(G) + 2\pi i/(p^f\log(q))$ have maximal order, if its order is at least $2$. 
\end{proof}

\begin{prop}\label{EpsilonMultL} Let $k\geq 2$. For any $k$ integers $1 \leq l_1, \ldots, l_k \leq p^e -1$, let
$$ \varepsilon(l_1,\ldots,l_k) = \alpha_p(G) - \frac{1 + \sum\nolimits_{i=1}^e r_i(G)\big(w_i(l_1) + \ldots + w_i(l_k)\big)}{l_1 + \ldots + l_k}. $$
Then we have $\varepsilon(l_1,\ldots,l_k) > 1/(2p^e)$.  
\end{prop}
\begin{proof} We first consider the following general situation. Let $q_i = a_i/b_i$ be positive fractions for $i=1,\ldots,k$ with $q_1 < q_2 < \ldots < q_k$, and let $q = (a_1 + \ldots + a_k)/(b_1 + \ldots + b_k)$. Then $q_1 < q < q_k$ holds. To see this, let $f_i(s) = a_i - b_is$ and $f(s) = f_1(s) + \ldots + f_k(s)$ be functions on real numbers. These functions are linearly decreasing and possess the zero $s = q_i$ and $s = q$ respectively. We have $f_1(s) \geq 0$ and $f_i(s) > 0$ for $i \neq 1$ and $s \leq q_1$. This implies $f(s) > 0$ for $s \leq q_1$. In the same way, we obtain $f(s) < 0$ for $s \geq q_k$. Thus the zero $s=q$ of $f(s)$ is located properly between $q_1$ and $q_k$. Together with this thought, Proposition \ref{EpsilonL} provides the chain of inequalities
\begin{align*}
\alpha_{p}(G) = \frac{1 + \sum\nolimits_{i=1}^e r_i(G)w_i(p^e)}{p^e} & \geq \max \left\{ \frac{1 + \sum\nolimits_{i=1}^e r_i(G) w_i(l_j)}{l_j} : 1\leq j \leq k \right\} \\
& \geq \frac{k + \sum\nolimits_{i=1}^e r_i(G) \big( w_i(l_1) + \ldots + w_i(l_k) \big)}{l_1 + \ldots + l_k}  \\
& = \alpha_{p}(G) - \varepsilon(l_1,\ldots,l_k) + \frac{k-1}{l_1 + \ldots + l_k}.
\end{align*}
Thus the proposed inequality follows from
$$ \varepsilon(l_1,\ldots,l_k) \geq \frac{k-1}{l_1 + \ldots + l_k} \geq \frac{k-1}{k(p^e-1)} > \frac{1}{2p^e}.  $$
\end{proof}

\begin{prop}\label{PsiG} The Euler product 
$$ \Psi_{G}(s) = \prod\limits_{\mf{p}} \Bigg( 1 + \sum\limits_{l=1}^{p^e-1} u_{G}\big(\mf{p}^l\big) \N{p}^{-ls} \Bigg) \prod\limits_{l=1}^{p^e-1} \Big( 1 - u_{G}\big(\mf{p}^l\big) \N{p}^{-ls} \Big) $$ 
is holomorphic for $\Re(s) > \alpha_p(G) - \varepsilon$ for some $\varepsilon > 1/(2p^e)$.
\end{prop}
\begin{proof} The Euler product $\Psi_{G}(s) = \prod\nolimits_{\mf{p}} \Psi_{G,\mf{p}}(s)$ converges absolutely if and only if all $\N{p}$-exponents of $\Psi_{G,\mf{p}}(\Re(s))$ are less $-1$. Let $s$ be a real number. If we recall the definition $U_l = u_G(\mf{p}^l)\N{p}^{-ls}$ from the proof of Proposition \ref{EulerFactors}, we obtain
\begin{equation}\label{Psih2}
 \Psi_{G,\mf{p}} = 1 + \sum\limits_{k=1}^{p^e} \sum\limits_{1 \leq l_1,\ldots,l_k \leq p^e -1} a(l_1,\ldots,l_k) U_{l_1} \cdots U_{l_k},
\end{equation}
where $a(l_1,\ldots,l_k)$ are suitable integers independent of $\mf{p}$. The sum corresponding to $k=1$ is zero, since the single-arguments $U_l$ occur twice with different sign respectively, i.e. we have $a(l)=0$ for $l=1,\ldots,p^e-1$. The $\N{p}$-exponents of multi-arguments $U_{l_1} \cdots U_{l_k}$ are of the form 
$$ \left( \sum\limits_{i=1}^e r_i(G) \big(w_i(l_1) + \ldots + w_i(l_k)\big) \right) - \big(l_1+\ldots + l_k\big) s. $$
These are less than $-1$ for 
$$ s > \frac{1 + \sum\nolimits_{i=1}^e r_i(G) \big(w_i(l_1) + \ldots + w_i(l_k)\big)}{l_1 + \ldots l_k}.    $$
Now Proposition \ref{PsiG} follows from Proposition \ref{EpsilonMultL}. 
\end{proof}

\begin{lemma}\label{MeromorphicContinuation} The Dirichlet series $\Phi_{G}(s)$ has the meromorphic continuation $\Phi_{G}(s) = \Psi_{G}(s) \Lambda_{G}(s)$ for $\Re(s) > \alpha_p(G) - \varepsilon$ and some $\varepsilon > 1/(2p^e)$.
\end{lemma}
\begin{proof} This is just a corollary of Proposition \ref{EulerFactors}, Proposition \ref{PropertiesLambdaG}, and Proposition \ref{PsiG}.
\end{proof}

\begin{prop}\label{AlphaGH} For any proper subgroup $pG \unlhd H \lhd G$, we have $\alpha_p(G) - \alpha_p(H) > 1/(2p^e)$ unless $r_e(G) = 1$, $r_e(H) = 0$, and $r_i(G) = r_i(H)$ for $i=1,\ldots,e-1$.
\end{prop}
\begin{proof} Let $\varepsilon(H) = \alpha_p(G) - \alpha_p(H)$. Note that $r_i(G) \geq r_i(H) \geq r_{i+1}(G)$ holds for $i\geq 1$, where the lower bound is a consequence of $pG \unlhd H$. We first assume $r_e(H) \neq 0$, such that $G$ and $H$ have the same exponent. Then we obtain
$$ \varepsilon(H) = \frac{\sum\nolimits_{i=1}^e (r_i(G) - r_i(H))w_i(p^e)}{p^e}, $$
by Proposition \ref{EpsilonL}. Thus $H \neq G$ implies $\varepsilon(H) \geq 1/p^e$ in this case. We now assume $r_e(H) = 0$. As $H$ contains $pG$, it hence has exponent $p^{e-1}$ and abscissa
$$ \alpha_p(H) = \frac{1 + \sum\nolimits_{i=1}^{e-1} r_i(H) w_i(p^{e-1})}{p^{e-1}} = \frac{p + \sum\nolimits_{i=1}^{e-1} r_i(H) w_i(p^e)}{p^e} = \frac{1 + w_e(p^e) + \sum\nolimits_{i=1}^{e-1} r_i(H) w_i(p^e)}{p^e}. $$
Note $w_e(p^e) = p-1$ and $pw_i(p^{e-1}) = w_i(p^e)$ hold for $i=1,\ldots,e-1$ hereby. This yields
$$ \varepsilon(H) = \frac{(r_e(G)-1)w_e(p^e) + \sum\nolimits_{i=1}^{e-1} (r_i(G) - r_i(H)) w_i(p^e)}{p^e}. $$
Hence $\varepsilon(H) = 0$ holds if and only if $r_e(G) = 1$ and $r_i(G) = r_i(H)$ hold for $i=1,\ldots,e-1$. Otherwise we have $\varepsilon(H) \geq 1/p^e$.
\end{proof}

\begin{proof}[\textbf{Proof of Theorem \ref{main1}.}] By Proposition \ref{PropertiesLambdaG}, we only have to show that $\Psi(s) = \Phi(F,G;s) \Lambda_{G}(s)^{-1}$ is convergent for $\Re(s) > a$ with $a = \alpha_p(G) - 1/(2p^e)$. We have the formal equality
$$ \Psi(s) = e_G \Psi_{G}(s) + \sum\limits_{G \neq H \unlhd G} e_H \Phi_{H}(s) \Lambda_{G}(s)^{-1} + \Upsilon_G(s) \Lambda_{G}(s)^{-1} $$
in the region of convergence, and we shall see that the right-hand side has convergence abscissa $a$. The zeros of $\Lambda_{G}(s)$ are located on the imaginary axes $\Re(s) = \big(1 + 2 \sum\nolimits_{i=1}^e r_i(G)w_i(l)\big)/2l$ for $l = 2,\ldots, p^e$, by the theorem of Hasse-Weil. These fractions are maximal for $l$ being maximal, whence $\Lambda_{G}(s)^{-1}$ has convergence abscissa $\big(1+2 \sum\nolimits_{i=1}^e r_i(G)w_i(p^e)\big)/2p^e = a$. Thus the same holds for $\Upsilon_G(s)\Lambda_{G}(s)^{-1}$ by Proposition~\ref{Upsilon}. Further is $\Psi_{G}(s)$ convergent for $\Re(s) > a$ by Proposition~\ref{PsiG}. Hence we are left to show that the sum $\sum\nolimits_{G \neq H \unlhd G} e_H \Phi_{H}(s) \Lambda_{G}(s)^{-1}$ is convergent for $\Re(s)>a$ as well. Firstly, the sum reduces to 
$$ \sum\limits_{G \neq H \unlhd G} e_H \Phi_{H}(s) \Lambda_{G}(s)^{-1} = \sum\limits_{pG \unlhd H \neq G} e_H \Phi_{H}(s) \Lambda_{G}(s)^{-1}, $$
since $e_H$ is zero for $H$ not containing $pG$, by definition of $e_H$ in Proposition~\ref{MoebiusFormula}. For the remaining summands, $\Phi_{H}(s)$ is convergent for $\Re(s)>a$ by Proposition~\ref{AlphaGH}, if it does not have convergence abscissa $\alpha_p(G)$. Otherwise, if $\Phi_{H}(s)$ has convergence abscissa $\alpha_p(G)$, then we have $r_i(H) = r_i(G)$ for $i = 1,\ldots, e-1$ and $r_e(H)=0$. This and the fact that $w_e(l)$ vanishes for $2\leq l\leq p^{e-1}$ implies
$$ \Lambda_{G}(s) = \prod\limits_{l=2}^{p^e} \zeta_F \Big( ls - \sum\limits_{i=1}^e r_i(G) w_i(l) \Big) = \Lambda_{H}(s) \prod\limits_{l=p^{e-1}+1}^{p^e} \zeta_F\Big( ls - \sum\limits_{i=1}^e r_i(G)w_i(l) \Big) $$
by the definition of $\Lambda$ in Proposition~\ref{PropertiesLambdaG}. Hence, the product 
$$ \Phi_{H}(s)\Lambda_{G}(s)^{-1} = \Psi_{H}(s)\prod\limits_{l=p^{e-1}+1}^{p^e} \zeta_F\Big( ls - \sum\limits_{i=1}^e r_i(G)w_i(l) \Big)^{-1} $$
is convergent for $\Re(s) > a$. Putting all together, $\Psi(s)$ has abscissa $a$.
\end{proof}

\section{Cauchy-Tauber formula}

Now we can finalise the proof of Theorem~\ref{main2}. We consider the Dirichlet series
$$ \varPhi(F,G;s) = \sum_{\mathrm{Gal}(E/F)\simeq G} \Norm{\mathfrak{f}(E/F)}^{-s} = \sum\limits_{n\geq 0} c_nt^n $$ 
as power series in $t=q^{-s}$ with coefficients $c_n$. Then the counting function $C(F,G;X)$ coincides with coefficients sum of this series via
$$ C(F,G;q^m) = \sum\limits_{n=0}^m c_n.  $$

\begin{proof}[\textbf{Proof of Theorem \ref{main2}.}] By Theorem \ref{main1}, the power series expansion of $\Phi(F,G;s)$ in $t=q^{-s}$ has radius of convergence $R=q^{-\alpha_p(G)}$, and it is meromorphic continuable beyond its circle of convergence. By Proposition \ref{PropertiesLambdaG}, the poles of the continuation are of the form $t=R\xi$ for finitely many roots of unity $\xi$. We now obtain Theorem \ref{main2} from a tauberian theorem like Theorem A.5 in \cite{Lagemann2012}.
\end{proof}


\textbf{Acknowledgement}. I was funded by Deutsche Forschungsgesellschaft via the priority project SPP 1489 KL 1424/8-1.

\begin{appendix}

\section{Delsarte formula}\label{Delsarte formula}

In this section, we develop an additive formula for the number of subgroups with given quotient, i.e. a formula for
$$ \kappa(A,G) = \left|\Big\{ B \unlhd A : A/B \iso G \Big\} \right|, $$
where $A$ is a finitely generated abelian group and $G$ is a finite abelian $p$-group. This shall be done with Jean Delsarte's theory of M\"obius functions for finite abelian groups \cite{Delsarte1948}. These methods was used in the mentioned article to compute a formula for the number $\lambda(A,G)$ of subgroups $B\unlhd A$ isomorphic with $G$ for finite $A$. This is nearly the same as computing $\kappa(A,G)$. For $A$ being finite, we have $\kappa(A,G) = \lambda(A,G)$ by the duality of abelian groups. The formula for $\kappa(A,G)$ extends to finitely generated $A$, and the formula for $\lambda(A,G)$ however does not. Albeit we can easily deduce the formula for $\kappa(A,G)$ from the formula for $\lambda(A,G)$ via $\kappa(A,G) = \kappa(A/\exp(G)A,G) = \lambda(A/\exp(G)A,G)$, I wish to recap the beautiful and short calculation of~$\kappa(A,G)$.

\begin{lemma}[See Delsarte \cite{Delsarte1948}]\label{DelsarteMoebiusLemma} Let $G$ be a finite abelian $p$-group with signature $\tupel{r}(G)$. Then there is a multivariate polynomial $f_G(\tupel{X})$ with degree $\tupel{r}(G)$ such that 
$$ f_G(\tupel{x}(A)) = |\Epi(A,G)| $$ 
equals the number of epimorphisms $A \rightarrow G$ for any finitely generated abelian groups $A$ with $\tupel{x}(A) = (\grindex{p^{i-1}A}{p^iA})_{i \geq 1}$. In particular, we have 
$$ f_G(\tupel{X}) = \sum\limits_{pG\unlhd H \unlhd G} \mu(G/H) \tupel{X}^{\tupel{r}(H)}, $$
where $\mu(A)$ denotes the Delsarte-M\"obius function for finite abelian $p$-groups with
$$ \mu(A) = \begin{cases} (-1)^r p^{r(r-1)/2}, & \text{if } A \simeq \cyclic{p}^r, r\geq 0, \\
            0, & \text{elsewise.} 
            \end{cases}
$$  
\end{lemma}

\begin{cor}\label{DelsarteMoebiusLemmaCorollary}  
\begin{enumerate}[(a)]
 \item The number of automorphisms $G \rightarrow G$ equals $f_G(\tupel{x}(G)) = |\Aut(G)|$.
 \item For any finitely generated abelian group $A$, we have $\kappa(A,G) = f_G(\tupel{x}(A))/|\Aut(G)|$. 
\end{enumerate}
\end{cor}

\begin{remark} In Jean Delsarte's work \cite{Delsarte1948}, we also find the factorisation 
$$ f_G(\tupel{X}) = \prod\limits_{i\geq 1} X_i^{r_{i+1}(G)} \prod\limits_{j=r_{i+1}(G)}^{r_i(G)-1} (X_i-p^j). $$
\end{remark}

\begin{proof}[Proof of Lemma \ref{DelsarteMoebiusLemma} and its corollary] All pieces for the proof are contained in \cite{Delsarte1948}, which I wish to reflect here. For any finitely generated abelian group $A$, let $\eta(A,G) = |\Hom(A,G)|$ be the number of homomorphisms $A \rightarrow G$, and let $\varepsilon(A,G) = |\Epi(A,G)|$ be the number of epimorphisms $A \rightarrow G$. Without loss of generality, we can assume $A$ to be finite, since all homomorphisms $A \rightarrow G$ factors through a common finite quotient group of $A$, i.e. $A/\exp(G)A$. These functions fulfil the relation
$$ \eta(A,G) = \sum\limits_{H \unlhd G} \varepsilon(A,H), $$
and the Delsarte-M\"obius inversion formula provides
$$ \varepsilon(A,G) = \sum\limits_{H \unlhd G} \mu(G/H) \eta(A,H) = \sum\limits_{pG \unlhd H \unlhd G} \mu(G/H) \eta(A,H). $$
The number of homomorphisms $A \rightarrow G$ depends only from the number of possible images for the elementary divisors of $A$. The images of the $r_i(A)-r_{i+1}(A)$ elementary divisors of $A$ equaling $p^i$ must be contained in $G[p^i]$, and any choice for their images in $G[p^i]$ is admissible. By duality of abelian groups, we have $G[p^i] \iso G/p^iG$, whence we obtain
\begin{align*} \eta(A,G) & = \prod\limits_{i\geq 1} |G[p^i]|^{r_i(A)-r_{i+1}(A)} = \prod\limits_{i\geq 1} p^{(r_1(G)+\ldots+r_i(G))(r_i(A)-r_{i+1}(A))} \\ & = \prod\limits_{i\geq 1} p^{r_i(A)r_i(G)} = \prod\limits_{i \geq 1} \grindex{p^{i-1}A}{p^iA}^{r_i(G)}. 
\end{align*}
Together with the Delsarte-M\"obius inversion formula above, we see that $\varepsilon(A,G)$ is a multivariate polynomial in $\tupel{X} = \tupel{x}(A)$ with the proposed assertions. In particular, Corollary \ref{DelsarteMoebiusLemmaCorollary} (b) follows from the well-known fact that $\Aut(G)$ acts faithfully on the set $\Epi(A,G)$ by composition, where the number of orbits equals the number of kernels of $\Epi(A,G)$. 
\end{proof}

\end{appendix}


\begin{thebibliography}{}

\bibitem{Delsarte1948}
  Delsarte, S., 1948, Fonctions de M\"obius Sur Les Groupes Abeliens Finis, Ann. of Math. 3, 600-609. 

\bibitem{Fesenko2002} 
  Fesenko, I., Vostokov, S., Local Fields and Their Extensions, Amer. Math. Soc., 2002. 

\bibitem{Hess2003} 
  He\ss, F., Pauli, M., Pauli, S., 2003. Computing the multiplicative group of residue class rings. Math. Comput. 72, 1531-1548.

\bibitem{KochAlgebraicNumberTheory}
 Koch, H., Algebraic Number Theory, Springer-Verlag, 1992.

\bibitem{KluenersMalle2004}
 Kl\"uners, J., Malle, G., 2004. Counting nilpotent Galois extensions. J. reine angew. Math. 572, 1 - 26.

\bibitem{LagemannAsymptotik}
 Lagemann, T., Asymptotik wild verzweigter abelscher Funktionenk\"orper, Logos Verlag, 2010.

\bibitem{Lagemann2012}
  Lagemann, T., 2012, Distribution of Artin-Schreier extensions, J. Number Theory 132, 1867-1887.



\bibitem{RosenNTFF}
  Rosen, M., 2002. Number Theory in Function Fields. Springer Verlag. 


\bibitem{StiALF}  
  Stichtenoth, H., 2008. Algebraic Function Fields and Codes, 2nd edition. Springer Verlag.

\bibitem{Wright1989}
  Wright, D., 1989, Distribution of discriminants of abelian extensions, Proc. London Math. Soc. (3) 58, 17-50.

\end{thebibliography}
\end{document}